\documentclass[11pt,a4paper]{amsart}
\usepackage[margin=3.5cm]{geometry}
\usepackage[english]{babel}
\usepackage[T1]{fontenc}
\usepackage{float}
\usepackage{latexsym,amsmath,amsfonts,amssymb,amsthm}
\usepackage[pdftex]{hyperref}%,doi,tikz,float,arydshln}
\usepackage{comment,makecell, array, euflag}
\newtheorem{theorem}{Theorem}[section]
\newtheorem{lemma}[theorem]{Lemma}
\newtheorem{proposition}[theorem]{Proposition}

\newtheorem{corollary}[theorem]{Corollary}
\theoremstyle{definition}

\newtheorem{example}[theorem]{Example}
\newtheorem{remark}[theorem]{Remark}

\newcommand{\SZ}{\mathbb{Z}}                    % Integers
\newcommand{\SQ}{\mathbb{Q}}                    % Rational numbers
\newcommand{\SC}{\mathbb{C}}                    % Complex numbers
\title{Rigid ideal sheaves and modular forms}
\author{\'Ad\'am Gyenge}
\address{Alfr\'ed R\'enyi Institute of Mathematics, Re\'altanoda utca 13-15, H-1053, Budapest, Hungary}
\email{Gyenge.Adam@renyi.hu}

\begin{document}
\maketitle

%\tableofcontents

\begin{abstract} Let $X$ be a complex smooth quasi-projective surface acted upon by a finite group $G$ such that the quotient $X/G$ has singularities only of ADE type. We obtain an explicit expression for the generating series of the Euler characteristics of the zero-dimensional components in the moduli space of zero-dimensional subschemes on $X$ invariant under the action of $G$. We show that this generating series (up to a suitable rational power of the formal variable) is a holomorphic modular form.
\end{abstract}

\section{Introduction}
\label{sec:intro}

%Hilbert schemes of points on surfaces have attracted a lot of attention. 

We investigate modular properties of a family of generating functions arising from certain enumerative invariants associated with complex smooth quasi-projective surface $X$ acted upon by a finite group $G$ such that the quotient $X/G$ has singularities only of ADE type. To put our results in context, denote by $\mathrm{Hilb}^m([X/G])$ the moduli space of $G$-invariant $0$-dimensional subschemes of $X$ of length $m$. This is the invariant part of the Hilbert scheme $\mathrm{Hilb}^m(X)$ of points\footnote{Another common notation for this object is
	$\mathrm{Hilb}^m(X)^G$. A justification for the two notations comes from the fact that a $G$ invariant, zero-dimensional subscheme of
	$X$ may equivalently be
	regarded as a zero-dimensional substack of $[X/G]$.} on $X$ under the lifted action of $G$. This Hilbert scheme is variously called the \textit{orbifold Hilbert scheme} \cite{young2010generating} or the \textit{equivariant Hilbert scheme} \cite{gusein2010generating}.

Let $\mathrm{RHilb}^m([X/G]) \subset \mathrm{Hilb}^m([X/G])$ be the \emph{rigid} part of the orbifold Hilbert scheme. This consists of those $G$-invariant ideal sheaves whose connected component in the orbifold Hilbert scheme is zero dimensional. Will will call this the \emph{rigid Hilbert scheme}\footnote{Similarly, the rigid Hilbert scheme can also be denoted as $\mathrm{RHilb}^m(X)^G$.} of the global quotient orbifold $[X/G]$.

We collect the topological Euler characteristics of these moduli spaces into generating functions. The \textit{$G$-fixed generating series} of $[X/G]$ is defined as
\[Z_{[X/G]}(q):= 1+\sum_{m=1}^{\infty}  \chi( \mathrm{Hilb}^n([X/G])) q^{m}.\]
%Here $q$ is a formal variable. 
The \textit{rigid $G$-fixed generating series} is defined as
\[R_{[X/G]}(q):= 1+\sum_{m=1}^{\infty}  \chi( \mathrm{RHilb}^n([X/G])) q^{m}.\]
In both series $q$ is a formal variable. 
\begin{comment}
Throughout this paper we set
\[
q=e^{2\pi i \tau},
\]
so that we may regard $Z_{\Delta}$ as a function of $\tau \in \mathbb{H}$ where $\mathbb{H}$ is the upper half-plane.

Our main result is the following.
\begin{theorem} There is rational number $\alpha$ such that $q^{\alpha}R_{[X/G]}(q)$ is a holomorphic modular form for a certain congruence subgroup.
\end{theorem}
\end{comment}

\subsection{The local case} The simplest example of a pair $(X,G)$ as above is a Kleinian (or simple) surface singularity orbifold $[\mathbb{C}^2/G_{\Delta}]$ where $G_{\Delta}<SL(2,\SC)$ is a finite subgroup. This follows from the known fact that finite subgroups of $SL(2,\SC)$ as well as the quotients $\mathbb{C}^2/G_{\Delta}$ have an ADE classification; the index $\Delta$ refers to the corresponding root system. This setup will have a distinguished role in our treatment; we will call it the  local setting. Then we have the local $G_{\Delta}$-fixed (resp. local rigid $G_{\Delta}$-fixed) generating series
\[Z_{[\mathbb{C}^2/G_{\Delta}]}(q)= 1+\sum_{m=1}^{\infty}  \chi( \mathrm{Hilb}^n([\mathbb{C}^2/G_{\Delta}])) q^{m}\]
%Here $q$ is a formal variable. 
and
\[R_{[\mathbb{C}^2/G_{\Delta}]}(q)= 1+\sum_{m=1}^{\infty}  \chi( \mathrm{RHilb}^n([\mathbb{C}^2/G_{\Delta}])) q^{m}.\]
%In both series $q$ is a formal variable. 

Let $\widetilde{\mathbb{C}^2/G_{\Delta}}$ be the minimal resolution of the Kleinian singularity $\mathbb{C}^2/G_{\Delta}$, which is known to be a smooth quasi-projective variety. Let $Z_{\widetilde{\mathbb{C}^2/G_{\Delta}}}(q)$ be the generating series of the Euler characteristics of the Hilbert scheme of points on $\widetilde{\mathbb{C}^2/G_{\Delta}}$ (see \cite{gottsche1990betti}).
Our first result might be known to experts, but we have not found it explicitly in the literature (however, see the closely related \cite[Lemma~5.2~(1)]{bryan2019g}).
\begin{theorem} 
\label{thm0}
Let $[\mathbb{C}^2/G_{\Delta}]$ be a Kleinian orbifold, and let $k=|G_{\Delta}|$, the order of $G_{\Delta}$. Then
\[ R_{[\mathbb{C}^2/G_{\Delta}]}(q)= \frac{Z_{[\mathbb{C}^2/G_{\Delta}]}(q)}{Z_{\widetilde{\mathbb{C}^2/G_{\Delta}}}(q^k)}.\]
\end{theorem}
Therefore, the series $R_{[\mathbb{C}^2/G_{\Delta}]}(q)$ can be thought of as the function measuring the extra information appearing in the Hilbert scheme when one replaces the classical resolution $\widetilde{\mathbb{C}^2/G}$ with the stack resolution $[\mathbb{C}^2/G]$.

\subsection{Modularity}
We will investigate modular properties of $R_{[X/G]}(q)$. Consider first the local setting. It is worth to consider a slightly corrected version of the generating series introduced above. Let
\[Z_{\Delta}(q):=q^{-\frac{1}{24}}(Z_{[\SC^2/G_\Delta]}(q))\]
and
\[R_{\Delta}(q):=q^{\frac{(n+1)k-1}{24}}(R_{[\SC^2/G_\Delta]}(q))\]
so that
\[ R_{\Delta}(q) = Z_{\Delta}(q)\cdot \left(q^{\frac{(n+1)k}{24}}\left({Z_{\widetilde{\mathbb{C}^2/G_{\Delta}}}(q^k)}\right)^{-1}\right). \]
Throughout this paper we set
\[
q=e^{2\pi i \tau},
\]
and therefore we may regard $Z_{\Delta}$ and $R_{\Delta}(q)$ as a function of $\tau \in \mathbb{H}$ where $\mathbb{H}$ is the upper half-plane.

%\begin{theorem}
%	\label{thm1}
% Let $\Delta$ be an ADE type root system of rank $n$. Then the function $R_\Delta(\tau)$ is a holomorphic modular form of weight $n/2$ for a certain congruence subgroup.
%\end{theorem}

%When the group action is trivial (that is, when $G_{\Delta}=1$), we get back the much investigated generating series of the Euler characteristics of the usual Hilbert scheme of points on $X$:
%\[Z(q)= q^{-\frac{1}{24}}(Z_{\SC^2}(q))=1+\sum_{n=1}^{\infty}  \chi( \mathrm{Hilb}^m(X)) q^{n-\frac{1}{24}}=q^{-1/24}\prod_{m\geq 1}(1-q^m)^{-1}=\frac{1}{\eta(\tau)}\]
%where $\eta(\tau)=q^{1/24}\prod_{m\geq 1}(1-q^m)$ is the Dedekind eta function \cite{gottsche1990betti}.

For $G<SL(2,\SC)$ an arbitrary finite subgroup, at least two explicit expressions are known for $Z_{\Delta}(\tau)$.
 First, in \cite[Theorem 1.3]{gyenge2018euler} we obtained an expression involving a sum over a rank $n$ lattice (see Theorem~\ref{thm:genfunct} below). %As we sketched in \cite[Lemma 4.2]{bryan2019g}, 
 This expression can be rewritten as
 \[ Z_{\Delta}(\tau)=\frac{\theta_\Delta(\tau)}{\eta(k\tau)^{n+1}} \]
  where $\theta_\Delta(\tau)$ is a shifted theta function over the root lattice of $\Delta$ (see \cite[Section 4]{bryan2019g} for a detailed treatment and for an expression for $\theta_\Delta(\tau)$). %We will give a complete proof of this result, and we will also show that
  %{\adam not completely precise, as there is a denominator.}
 %We will now reproduce this with more details.
Second, in \cite[Theorem 1.2]{bryan2019g} the same generating series was expressed as 
 \[ Z_{\Delta}(\tau)=\frac{\eta_\Delta(\tau)}{\eta(k\tau)^{n+1}}. \]
Here $\eta_\Delta(\tau)$ is a product of scaled Dedekind eta functions with (possibly negative) integer powers. Such expressions are generally called eta products \cite{kohler2011eta}.

Using the classical formula of Göttsche \cite{gottsche1990betti},
\[ q^{\frac{(n+1)k}{24}}\left({Z_{\widetilde{\mathbb{C}^2/G_{\Delta}}}(q^k)}\right)^{-1}=q^{\frac{(n+1)k}{24}}\left(\prod_{m=0}^{\infty}(1-q^{km})^{n+1}\right)^{-1}=\eta(k\tau)^{n+1}\]
where $\eta(\tau)$ is the Dedekind eta function.
Combining this with Theorem~\ref{thm0}, we have two expressions for $R_{\Delta}(\tau)$:
\[ R_{\Delta}(\tau)=\theta_\Delta(\tau),\quad \textrm{and}\quad R_{\Delta}(\tau)=\eta_\Delta(\tau).\] 
The equality of the two expressions on the right sides was already observed in \cite[Section 4]{bryan2019g}.
%In fact, $Z_{\Delta}(\tau)$ is itself an eta product.
%This implies that $Z_{\Delta}(q)$ transforms as a modular form for some congruence subgroup. 

%In both expressions the denominator $\eta(k\tau)^{n+1}$ corresponds to $q^{-\frac{1}{24}} Z_{\widetilde{\mathbb{C}^2/G}}(q)$ where $\widetilde{\mathbb{C}^2/G}$ is the minimal resolution of the Kleinian singularity $\mathbb{C}^2/G$ and $Z_{\widetilde{\mathbb{C}^2/G}}(q)$ is the Euler characteristic generating series of the Hilbert scheme of points on $\widetilde{\mathbb{C}^2/G}$ (see \cite{gottsche1990betti}).

The function $Z_{\Delta}(\tau)$ is known to be a meromorphic modular form \cite{toda2015s, gyenge2018euler}. This fact can be interpreted as a counterpart of the S-duality conjecture of Vafa--Witten \cite{vafa1994strong} to ADE orbifolds. 
One can verify that $\eta(k\tau)^{n+1}$ appearing in the denominator of $Z_{\Delta}(\tau)$  is a holomorphic modular form (see~Section~\ref{sec:locetaprod}). It follows that $R_\Delta(\tau)$ is also a meromorphic modular form. Our second aim is to investigate in-depth its modular properties. Surprisingly, it turns out to be holomorphic in each ADE case. While the proof of Theorem~\ref{thm0} is based on the identity $R_{\Delta}(\tau)=\theta_\Delta(\tau)$, the modular properties of $R_\Delta(\tau)$ will be determined using the identity $R_{\Delta}(\tau)=\eta_\Delta(\tau)$. 
The relevant congruence subgroup will turn out to be
\[  \Gamma_0(N) = \left\{ \begin{pmatrix} a & b \\ c & d \end{pmatrix} \in SL(2,\SZ) : \begin{pmatrix} a & b \\ c & d \end{pmatrix} \equiv \begin{pmatrix} \ast & \ast \\ 0 & \ast \end{pmatrix} \;(\mathrm{mod}\;N) \right\}.\]
\begin{theorem}	
\label{thm:maineta}
 Let $\Delta$ be an ADE type root system of rank $n$ and let $k=|G_{\Delta}|$. The function $\eta_\Delta(\tau)$, and hence also $R_\Delta(\tau)$, is a holomorphic modular form of weight $n/2$ for $\Gamma_0(k)$ with the multiplier system $\chi_{\Delta}(A)$ given in Corollary~\ref{cor:multetaD} below. %In particular, $\eta_\Delta(\tau)$ is never a cusp form.
\end{theorem}
%Here $\Gamma(km,k)$ and $\Gamma_0(k)$ are certain congruence subgroups of ... introduced in ... below.
%These two results are compatible with each other in the following sense. As $\Gamma_0(k) \cap \Gamma(km,k) = \Gamma(km,k)$, the two multiplier systems must agree over $\Gamma(km,k)$. For the same reason, Theorem~\ref{thm:maineta} is a bit stronger. Still, modular forms arising from theta functions have often additional properties. This is why we include both proofs.

\begin{comment}
The advantage of eta products is that their analysis often reduces to combinatorial calculations. One example is the order of vanishing at the cusps of a congruence subgroup. As mentioned in Section~\ref{sec:locetaprod} below, to verify that $\eta_\Delta(\tau)$ is a holomorphic modular form for $\Gamma_0(k)$ it is enough to calculate these orders at cusps of the form $\frac{1}{c}$ where $c$ is a positive divisor of $k$. For the purpose of the global calculations, we will prove a slightly stronger statement.
\begin{proposition} Let $\Delta$ be of ADE type, and let $k=|G_{\Delta}|$. Then $\eta_\Delta(\tau)$ is a 
\end{proposition}

Along the way to prove Theorem~\ref{thm:maineta} we obtain the following.
\begin{corollary}
The order of $\eta_\Delta(\tau)$, and hence also of $R_\Delta(\tau)$, at every cusp of $\Gamma_0(k)$ except for 1 is positive. Its order at the cusp 1 is 0.
\end{corollary}
\end{comment}

\subsection{The global case}
Using the motivic property of Hilbert schemes we can globalize the results above. Let now $X$ be an arbitrary smooth quasi-projective surface acted upon by a finite group $G$ such that the quotient $X/G$ has only ADE singularities. Let $p_1,\dots,p_r \in X/G$ be the collection of singular points of the quotient. These determine a collection of ADE type root systems $\Delta_{1},\dots,\Delta_{r}$. Let $k=|G|$, $k_i=|G_{\Delta_i}|$ and $n_i$ be the rank of $\Delta_{i}$. %, $\alpha_i=(n_i+1)k_i-1$ and $\alpha=\sum_{i=1}^r \alpha_i$.%associated to $(X,G)$
\begin{theorem} 
\label{thm:glob}	
With the above notation we have
\[ R_{[X/G]}(q)=\prod_{i=1}^r R_{[\mathbb{C}^2/G_{\Delta_i}]}(q^{k/k_i}).\]
\end{theorem}

As \[R_{\Delta_i}(\tau)=q^{\frac{(n_i+1)k_i-1}{24}}R_{[\mathbb{C}^2/G_{\Delta_i}]}(q)\] is a holomorphic modular form for $\Gamma_0(k_i)$, it follows from  Lemma~\ref{lem:higherlevel} below that \[R_{\Delta_i}\left(\frac{k}{k_i}\tau\right)=q^{\frac{k((n_i+1)k_i-1)}{24k_i}}R_{[\mathbb{C}^2/G_{\Delta_i}]}(q^{k/k_i})\] is a holomorphic modular form for \[\Gamma_0\left(k_i\cdot\frac{k}{k_i}\right)=\Gamma_0(k).\]
Since the product of modular forms for $\Gamma_0(k)$ is a modular form for $\Gamma_0(k)$ \cite[p. 17]{diamond2005first}, we obtain the following global modularity result.
\begin{corollary} 
\label{cor:globmod}	
With the above notation the function
\[\begin{aligned}q^{\sum_{i=1}^r\frac{k((n_i+1)k_i-1)}{24k_i}}\cdot R_{[X/G]}(q) & =\prod_{i=1}^r q^{\frac{k((n_i+1)k_i-1)}{24k_i}}R_{[\mathbb{C}^2/G_{\Delta_i}]}(q^{k/k_i})\\ & =\prod_{i=1}^r R_{\Delta_i}\left(\frac{k}{k_i}\tau\right)
\end{aligned}\]
is a holomorphic modular form for $\Gamma_0(k)$ of weight $\frac{1}{2}\sum_{i=1}^r n_i$.
\end{corollary}
It is worth to compare this with \cite[Theorem 1.1]{bryan2019g}, which shows that $q(Z_{[X/G]}(q))^{-1}$ is a holomorphic modular form provided that $X$ is a K3 surface (a similar statement for abelian surfaces was obtained recently in \cite{pietromonaco2020g}). Note however that in Corollary~\ref{cor:globmod}	we do not take reciprocal.

\subsection{Further remarks and the structure of the paper}
As the moduli of rigid ideal sheaves consists of isolated points, the Euler numbers $\chi(\mathrm{RHilb}^n([X/G]))$ count actually the points in $\mathrm{RHilb}^n([X/G])$. Hence, all rigid generating series also enumerate the appropriate classes (multiples of the point) in the Grothendieck ring of varieties (over $\SC$).

%Our results have a connection to combinatorics. It was shown in \cite{gyenge2018euler} that when $\Delta$ is of type A or D, the series $Z_{[\SC^2/G_\Delta]}(q)$ enumerates certain combinatorial objects called Young walls. A distinguished subset of the Young walls is that of core Young walls.

The structure of the paper is the following. In Section 2, we analyze the local and global rigid generating series. In particular, we prove Theorems~\ref{thm0} and \ref{thm:glob}. In Section 3, after reviewing the basics of eta products, we prove Theorem~\ref{thm:maineta}. In the Appendix the orders of $R_{\Delta}(\tau)$ are collected when $\Delta$ is of type E.

\subsection*{Acknowledgement} The author would like to thank to Jim Bryan for helpful comments and discussions.
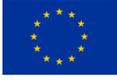
\begin{center}
	\begin{tabular}{p{1.5cm} p{11.5cm}}
		 \vspace{-0.2cm} \resizebox{1.5cm}{!}{\Huge  \euflag} 
		& 
		This project has received funding from the European Union's Horizon 2020
		research and innovation programme under the Marie Sk\l odowska-Curie grant
		agreement No 891437.
	\end{tabular}
\end{center}

\section{Rigid ideal sheaves}

\subsection{Local calculations: the proof of Theorem~\ref{thm0}}
\label{subsec:loccalc}
Let $G_{\Delta} < SL(2,\SC)$ be a finite subgroup such that $\Delta$ is a root system of rank $n$. Let $k=|G_{\Delta}|$, the order of $G_{\Delta}$.
%A zero-dimensional substack $Z\subset [\SC^{2}/G_{\Delta}]$ may be
%regarded as a $G_{\Delta}$ invariant, zero-dimensional subscheme of
%$\SC^{2}$. This correspondence identifies naturally the Hilbert scheme of points
%on the stack $[\SC^{2}/G_{\Delta}]$ with the $G_{\Delta}$ fixed locus
%of the Hilbert scheme of points on $\SC^{2}$: 
%\[
%\mathrm{Hilb} \left([\SC^{2}/G_{\Delta}] \right) = \mathrm{Hilb}
%([\mathbb{C}^2/G_{\Delta}]) .
%\]

It is known that $\mathrm{Hilb}([\mathbb{C}^2/G_{\Delta}]) $ decomposes into disjoint subvarieties
\begin{equation*} 
%\label{eq:GDdec}	
\mathrm{Hilb}([\mathbb{C}^2/G_{\Delta}])=\bigsqcup_{\rho \in {\mathop{\rm Rep}}({G_{\Delta}})}\mathrm{Hilb}^{\rho}([\mathbb{C}^2/G_{\Delta}]),\end{equation*}
where
\[\mathrm{Hilb}^{\rho}([\mathbb{C}^2/G_{\Delta}])=  \{ I \in \mathrm{Hilb}([\mathbb{C}^2/G_{\Delta}]) \colon H^0(\mathcal{O}_{\SC^2}/I) \simeq_{G_{\Delta}} \rho \}\]
for any finite-dimensional representation $\rho\in {\mathop{\rm Rep}}(G_{\Delta})$ of $G_{\Delta}$ (see \cite{gyenge2018euler}).
Denote 
\[ \mathrm{RHilb}^{\rho}([\mathbb{C}^2/G_{\Delta}]):=\mathrm{RHilb}([\mathbb{C}^2/G_{\Delta}])\cap\mathrm{Hilb}^{\rho}([\mathbb{C}^2/G_{\Delta}]). \]

There correspond multivariable generating series to these varieties:
\[Z_{[\SC^2/G_{\Delta}]}(q_0,\ldots, q_n):= \sum_{m_0,\dots,m_n=0}^\infty \chi\left(\mathrm{Hilb}^{m_0 \rho_0 + \ldots +m_n \rho_n}([\SC^2/G_{\Delta}]) \right)   q_0^{m_0}\cdot \ldots \cdot q_n^{m_n},\]
\[R_{[\SC^2/G_{\Delta}]}(q_0,\ldots, q_n):= \sum_{m_0,\dots,m_n=0}^\infty \chi\left(\mathrm{RHilb}^{m_0 \rho_0 + \ldots +m_n \rho_n}([\SC^2/G_{\Delta}]) \right)   q_0^{m_0}\cdot \ldots \cdot q_n^{m_n}\]
where $\{\rho_{0},\dotsc ,\rho_{n} \}$ are the irreducible
representations of $G_{\Delta}$ with $\rho_{0}$ the trivial
representation. We note that $n$ is also the rank of $\Delta$.

One recovers the one variable (rigid or non-rigid) $G_{\Delta}$-fixed generating series of from the above multivariable generating series as follows.

\begin{lemma} %The G-fixed generating series of $[\SC^2/G]$ is the following specialization of the orbifold generating series:
	\label{lem:thetasubst}
\begin{enumerate}
\item 	\[Z_{[\SC^2/G_{\Delta}]}(q)=Z_{[\SC^2/G_{\Delta}]}(q_0,\ldots, q_n)\Big|_{q_i=q^{\mathrm{dim} \rho_i}}.\]
\item    \[R_{[\SC^2/G_{\Delta}]}(q)=R_{[\SC^2/G_{\Delta}]}(q_0,\ldots, q_n)\Big|_{q_i=q^{\mathrm{dim} \rho_i}}.\]
\end{enumerate}
\end{lemma}
\begin{proof}
	Let $I$ be an equivariant ideal such that $H^0(\mathcal{O}_{\SC^2}/I) \simeq_{G_{\Delta}} \rho$, where $\rho \simeq m_0 \rho_0 + \ldots +m_n \rho_n$. This implies that
	\[\mathrm{dim} H^0(\mathcal{O}_{\SC^2}/I) = \sum_{i=0}^{n} m_i \mathrm{dim} \rho_i, \]
	when $I$ is considered as a non-equivariant ideal of $\SC^2$.
\end{proof}

The orbifold generating series of a simple singularity orbifold is given explicitly by the following result.
%As we recall in Appendix~\ref{sect:aff_lie_hilb}, the following result is known. 
\begin{theorem}[\cite{nakajima2002geometric, gyenge2018euler}]  
	\label{thm:genfunct}
	Let $[\SC^2/G_\Delta]$ be a Kleinian orbifold. Then 
	\begin{multline*} Z_{[\SC^2/G_\Delta]}(q_0,\dots,q_n)=\left(\prod_{m=1}^{\infty}(1-\mathbf{q}^m)^{-1}\right)^{n+1} \\
		\cdot\sum_{ \mathbf{m}=(m_1,\dots,m_n) \in \SZ^n } q_1^{m_1}\dots q_n^{m_n}(\mathbf{q}^{1/2})^{\mathbf{m}^\top \cdot C_\Delta \cdot \mathbf{m}},\label{eq:orbi_main_formula}\end{multline*}
	where $\mathbf{q}=\prod_{i=0}^n q_i^{\dim\rho_i}$ and $C_\Delta$ is the finite type Cartan matrix corresponding to $\Delta$.
\end{theorem}
Göttsche's formula \cite{gottsche1990betti} applied on the resolution $\widetilde{\mathbb{C}^2/G_{\Delta}}$ shows that when we substitute $q_i=q^{\mathrm{dim} \rho_i}$ for each $0 \leq i \leq n$, the first term of this expression gives precisely $Z_{\widetilde{\mathbb{C}^2/G_{\Delta}}}(q^k)$
due to the fact that 
\[\sum_{i=0}^n (\mathrm{dim} \rho_i)^2=|G_{\Delta}|=k.\]
It turns out from the next result that the second term in Theorem~\ref{thm:genfunct} 
gives exactly the rigid generating series. From this and the previous observation we obtain Theorem~\ref{thm0}.
\begin{proposition}[{Compare with \cite[Lemma~5.2~(1)]{bryan2019g}}]
\label{prop:RD}
\[ R_{[\SC^2/G_\Delta]}(q_0,\dots,q_n)=\sum_{ \mathbf{m}=(m_1,\dots,m_n) \in \SZ^n } q_1^{m_1}\dots q_n^{m_n}(\mathbf{q}^{1/2})^{\mathbf{m}^\top \cdot C_\Delta \cdot \mathbf{m}}\]
\end{proposition}
\begin{proof}
%Let $\SZ \Delta_{\mathrm{fin}} \cong \SZ^n$ be the root lattice of $\Delta$.
%Cf. \cite[Lemma~5.2~(1)]{bryan2019g}. Here we write out the details. %it was shown that to every $ \mathbf{m} \in \SZ \Delta$ there corresponds a unique rigid 

Nakajima has shown \cite[Section~2]{nakajima2002geometric} that 
\[
\mathrm{Hilb}^{v_{0}\rho_{0}+\dotsb +v_{n}\rho_{n}}([\SC^{2} /G_{\Delta}]) = M(\mathbf{v} ,\mathbf{w} )
\]
where $\mathbf{w}  = (1,0,\dotsc ,0)$ and $M(\mathbf{v} ,\mathbf{w}  )$ is the Nakajima
quiver variety associated to the affine Dynkin diagram of $\Delta$
with framing vector $\mathbf{w} $ and dimension vector $\mathbf{v}$. 

Let $	\delta = (1,\mathrm{dim}\,\rho_0,\dotsc ,\mathrm{dim}\,\rho_n) \in \mathbb{N}^{n+1}$, which can also be identified with the basic imaginary root of the affine root system corresponding to $\Delta$.
Then every $\mathbf{v} \in \mathbb{N}^{n+1}$ decomposes uniquely as 
\begin{equation} 
	\label{eq:vkdelta}
	\mathbf{v}=k\delta+\frac{1}{2}(\mathbf{m}|\mathbf{m})\delta+(0,\mathbf{m})
\end{equation} 
where $k \in \mathbb{Z}$, $\mathbf{m} \in \mathbb{Z}^{n}\cong \mathbb{Z}\Delta$ is an element of the finite root lattice and $(\cdot|\cdot)$ is the inner product with respect to the finite Cartan matrix. For this, write $v-v_0\delta$ as $(0,\mathbf{m})$. Then $k=v_0-\frac{1}{2}(\mathbf{m}|\mathbf{m})$, because $\delta_0=1$ always.

By
\cite[(2.6)]{nakajima1994instantons} we have 
\begin{align*}
	\dim M(\mathbf{v} ,\mathbf{w}  )& = 2\mathbf{v} \cdot \mathbf{w}  -\left\langle \mathbf{v} ,\mathbf{v} \right\rangle\\
	&= 2v_{0} - \left\langle \mathbf{v} ,\mathbf{v} \right\rangle\\
	&=2k+ (\mathbf{m} |\mathbf{m} )  - \left\langle \mathbf{v} ,\mathbf{v}\right\rangle
\end{align*}
where $\left\langle \cdot , \cdot  \right\rangle$ is the inner product
given by the Cartan matrix associated with the affine Dynkin diagram.

Moreover,
\[
\left\langle (0,\mathbf{m}) ,(0,\mathbf{m}) \right\rangle = (\mathbf{m} |\mathbf{m} ),\quad
\quad \left\langle \delta ,\delta \right\rangle = 0, \quad \quad
\left\langle (0,\mathbf{m}) ,\delta \right\rangle = 0.
\]
The first follows directly from our definitions, and the later two are
well known properties of the vector $\delta$.  Using these we
compute the dimension of corresponding component of the Hilbert scheme:
\begin{gather*}
	\dim M(\mathbf{v} ,\mathbf{w} )  \\
	=2k+(\mathbf{m}| \mathbf{m}) - \left\langle
\left(k+ \frac{1}{2}( \mathbf{m}| \mathbf{m} )\right)\delta+(0,\mathbf{m}) ,\left(k+ \frac{1}{2}( \mathbf{m}| \mathbf{m} )\right) \delta+(0,\mathbf{m})  \right\rangle \\
	 =2k+ (\mathbf{m}| \mathbf{m} )  -\left(k+ \frac{1}{2}( \mathbf{m}| \mathbf{m} )\right)^2\langle \delta, \delta \rangle \\ - 2\left(k+ \frac{1}{2}( \mathbf{m}| \mathbf{m} )\right) \langle (0,\mathbf{m}), \delta \rangle  - \langle (0,\mathbf{m}), (0,\mathbf{m})\rangle\\
	= 2k.
\end{gather*}

Therefore, the component $\mathrm{Hilb}^{v_{0}\rho_{0}+\dotsb +v_{n}\rho_{n}}([\SC^{2} /G_{\Delta}])$ is of zero dimension if and only if $k=0$ in the decomposition \eqref{eq:vkdelta}. There is exactly one such component for each $\mathbf{m} \in \mathbb{Z}^{n}\cong \mathbb{Z}\Delta$, to which there corresponds the term
\[ q_1^{m_1}\dots q_n^{m_n}(\mathbf{q}^{1/2})^{( \mathbf{m}| \mathbf{m})}=q_1^{m_1}\dots q_n^{m_n}(\mathbf{q}^{1/2})^{\mathbf{m}^\top \cdot C_\Delta \cdot \mathbf{m}}\]
in the generating series. 
%in Theorem~\ref{thm:genfunct}. %because
%\[\frac{1}{2}( \mathbf{m}| \mathbf{m})=\frac{1}{2} (\mathbf{m}^\top \cdot C_\Delta \cdot \mathbf{m})\]
\end{proof}

\subsection{Global calculations: the proof of Theorem~\ref{thm:glob}}

To  globalize our results so far we perform a similar calculation as in \cite[Section 2]{bryan2019g}, but we replace the Hilbert scheme with the rigid Hilbert scheme.

As in the introduction, let $X$ be a smooth quasi-projective surface with a symplectic
action of a finite group $G$.  Recall that $p_{1},\dotsc ,p_{r}\in
X/G$ are the singular points of $X/G$. To these there correspond the stabilizer
subgroups $G_{i}\subset G$ of order $k_{i}$ and ADE type
$\Delta_{i}$. Let $\{x_{i}^{1},\dotsc ,x_{i}^{k/k_{i}} \}$ be the
orbit of $G$ in $X$ corresponding to the point $p_{i}$ (recall that
$k=|G|$).  We may stratify $\mathrm{RHilb}([X/G])$ according to the orbit
types of subscheme as follows.

Suppose $Z\subset X$ is a rigid $G$-invariant subscheme of length $nk$ whose
support lies on free orbits. Then $Z$ determines and is determined by
a rigid length $n$ subscheme of 
\[
(X/G)^{o}  = X/G\setminus \{p_{1},\dotsc ,p_{r} \},
\]
i.e. a point in $\mathrm{RHilb}^{n}((X/G)^{o})$. 
But as $(X/G)^{o}$ is smooth, this implies that $n=0$. Hence, $nk=0$ as well, and $Z$ is empty.

On the other hand, suppose $Z\subset X$ is a rigid $G$-invariant subscheme
of length $\frac{nk}{k_{i}}$ supported on the orbit
$\{x_{i}^{1},\dotsc ,x_{i}^{k/k_{i}} \}$. Then $Z$ determines and is
determined by the length $n$ component of $Z$ supported on a formal
neighborhood of one of the points, say $x_{i}^{1}$. Choosing a
$G_{i}$-equivariant isomorphism of the formal neighborhood of
$x_{i}^{1}$ in $X$ with the formal neighborhood of the origin in
$\SC^{2}$, we see that $Z$ determines and is determined by a point in
$\mathrm{RHilb}_{0}^{n}([\mathbb{C}^2/G_{i}])$, 
the rigid Hilbert scheme parameterizing rigid subschemes supported on a formal neighborhood of the origin in
$\SC^{2}$. As all such rigid subschemes are supported on the origin itself,
\[ \mathrm{RHilb}_{0}^{n}([\mathbb{C}^2/G_{i}]) \cong \mathrm{RHilb}^{n}([\mathbb{C}^2/G_{i}]).\]

By decomposing an arbitrary $G$-invariant subscheme into components of
the above types, we obtain a stratification of $\mathrm{RHilb} ([X/G])$ into
strata which are given by products of
$\mathrm{RHilb}([\mathbb{C}^2/G_{1}]),\dotsc ,\mathrm{RHilb}([\mathbb{C}^2/G_{r}])$. Then
using the fact that Euler characteristic is additive under
stratifications and multiplicative under products, we obtain the
following equation of generating functions:
\begin{equation}\label{eqn: stratification formula for sum e(hilb(X)G)}
	\nonumber\sum_{n=0}^{\infty} \chi\left(\mathrm{RHilb}^{n}([X/G]) \right)\, q^{n}
	 = \prod_{i=1}^{r}\left( \sum_{n=0}^{\infty}
	\chi\left(\mathrm{RHilb}^{n}([\mathbb{C}^2/G_{i}]) \right) \, q^{\frac{nk}{k_{i}}}
	\right) .
\end{equation}
This proves Theorem~\ref{thm:glob}.

\section{Eta products}
\label{sec:locetaprod}

\subsection{Review of modular forms and eta products}

We will work with modular forms of possibly half-integer weight. Fix a subgroup $\Gamma$ of finite index in $SL(2,\SZ)$, a function $\chi \colon \Gamma \to \SC^{\ast}$ with $|\chi(A)|=1$ for $A \in \Gamma$, and a half-integer $k$. Then a holomorphic function $f \colon \mathbb{H} \to \SC$ is said to transform as a modular form of weight $k$ with the multiplier system $\chi$ for $\Gamma$ if
\[ f\left( \frac{a\tau+b}{c\tau+d}\right)=\chi(A)(c\tau+d)^k f(\tau) \quad \textrm{ for all } A=\begin{pmatrix}
	a & b \\ c & d
\end{pmatrix} \in \Gamma. \]
When $k$ is not an integer, $(c\tau+d)^k$ is understood to be a principal value. If moreover $f$ is holomorphic at all the cusps of $\Gamma$ on $\SQ \cup \{ \infty\}$, then $f$ is said to be a modular form. The space of modular forms of weight $k$ and multiplier system $\chi$ for $\Gamma$ is denoted by $M_k(\Gamma, \chi)$.

An eta product is a finite product
\begin{equation} f(\tau)=\prod_m \eta(m \tau )^{a_m} 
\label{eq:etaproddef}
\end{equation}
where $m$ runs through a finite set of positive integers and the exponents
$a_m$ may take values from $\SZ$. The least common multiple of all $m$ such that $a_m\neq 0$ will be denoted by $N$; it is called the level of $f(\tau)$. 

Eta products are known to transform as a modular form for $\Gamma_0(N)$ of weight \[\frac{1}{2}\sum_m a_m.\] 

The multiplier system of the eta function is given by a formula of Petersson. For this we need some notation. Let $\mathrm{sgn}(x)=\frac{x}{|x|}$ be the sign of a real number $x \neq 0$. For short, we will write $e(z)$ for the function $e^{2 \pi i z}$ with $z \in \mathbb{C}$. Let $c$ and $d$ be integers such that their greatest common divisor $(c,d)=1$, $d$ is odd and $c \neq 0$. Denote with $\left(\frac{c}{d}\right)$ their Legrende-Jacobi-Kronecker symbol. Then let
\[ \left(\frac{c}{d}\right)^\ast=\left(\frac{c}{|d|}\right)\quad \textrm{and} \quad\left(\frac{c}{d}\right)_\ast= \left(\frac{c}{|d|}\right)\cdot (-1)^{\frac{1}{4}(\mathrm{sgn}(c)-1)(\mathrm{sgn}(d)-1)}.\]
Put furthermore
\[  \left(\frac{0}{1}\right)^\ast=\left(\frac{0}{-1}\right)^\ast=1, \quad \left(\frac{0}{1}\right)_\ast=1, \quad \left(\frac{0}{-1}\right)_\ast=-1. \]
Then the multiplier system of the eta function is given in \cite[Section 4.1]{knopp2008modular}, \cite[Section 1.3]{kohler2011eta}:
\begin{equation}
\label{eq:veta}
v_\eta(A)=\begin{cases}\left(\frac{d}{c}\right)^{\ast}e(\frac{1}{24}((a+d)c -bd(c^2-1)-3c)), & \textrm{if }c\textrm{ is odd,}\\
\left(\frac{c}{d}\right)_{\ast}e(\frac{1}{24}((a+d)c -bd(c^2-1)+3d-3-3cd)), & \textrm{if }c\textrm{ is even.}
\end{cases}
\end{equation}
Using this, the multiplier system of the eta product \eqref{eq:etaproddef} is (see \cite[Section 2.1]{kohler2011eta}):
\[v_f(A)=\prod_{m}\left(v_\eta\begin{pmatrix}
a & mb \\ c/m & d
\end{pmatrix}\right). \]
For the eta product %$f(\tau)$ %as in 
\eqref{eq:etaproddef} 
we will denote by
\[ (f)^\ast(A):=\prod_{m} \left(\left(\frac{d}{c/m}\right)^{\ast}\right)^{a_m} \]
and
\[ (f)_\ast(A):=\prod_{m} \left(\left(\frac{c/m}{d}\right)_{\ast}\right)^{a_m} \]
the terms of the multiplier system $v_f(A)$ coming from the extended Kronecker symbol when $c$ is respectively odd or even.
Using this notation and expression \eqref{eq:veta} the multiplier system $v_f(A)$ of the eta product \eqref{eq:etaproddef} can be written explicitly as follows.
\begin{corollary}
\label{cor:fmultexp}
\[ v_f(A)=\begin{cases}
	\begin{split}
	(f)^\ast(A)e\left(\frac{1}{24}\left(((a+d)c -bdc^2-3c)\left(\sum_m \frac{a_m}{m}\right)+bd\left(\sum_m ma_m\right)\right)\right), \\ \textrm{if }c\textrm{ is odd,}
	\end{split}
	\\
	\begin{split}
	(f)_\ast(A)e\Bigg(\frac{1}{24}\Bigg(((a+d)c -bdc^2-3cd)\left(\sum_m \frac{a_m}{m}\right)+bd\left(\sum_m ma_m\right)\\
	+3(d-1)\left(\sum_m a_m\right)\Bigg)\Bigg), \\ \textrm{if }c\textrm{ is even.}
\end{split}
\end{cases}
\]
\end{corollary}

Since $\eta(\tau)$ is nonzero on $\mathbb{H}$, an eta quotient never has finite poles. The only issue for an eta product to be a modular form is whether the numerator vanishes to at least the same order as the denominator at each cusp. Recall that cusps of $\Gamma_0(N)$ are in bijection with the orbits of $\Gamma_0(N)$ on the set $\mathbb{Q} \cup \{ \infty\}$. 
\begin{lemma}[{\cite[Chapter 2]{kohler2011eta}}]
\label{lem:cusporder}
Let $f$ be an eta product as in \eqref{eq:etaproddef} and let $N$ be its level. 
\begin{enumerate}
\item  The order of $f$ at the cusp $\infty$ is:
\[\mathrm{ord}(f,\infty)=\frac{1}{24}\sum_{m} ma_m.\]
\item\label{it:cusporderdiv} The order of $f$ at the cusp $r=-\frac{d}{c} \in \mathbb{Q}$, $(c,d)=1$ is
\[ \mathrm{ord}(f,r)=\frac{1}{24}\sum_{m} \frac{(c,m)^2}{m}a_m.\]
\end{enumerate}
\end{lemma}
As  $\mathrm{ord}(f,r)$ at $r=-\frac{d}{c}$ depends only on the denominator $c$, it is in fact enough to check the order of vanishing at cusps of the form $\frac{1}{c}$ where $c$ is a divisor of $N$. In other words, the data of orders at all cusps can be reconstructed from the data of orders at the cusps $\{ \frac{1}{c}: c|N,\, c> 0\}$.
Applying this observation, the following result gives an explicit  condition for an eta product to be a modular form.
\begin{proposition}[{\cite[Corollary 2.3]{kohler2011eta}}] An eta product $f$ as in \eqref{eq:etaproddef} is holomorphic for $\Gamma_0(N)$ if and only if the inequalities
\[ \sum_{m} \frac{(c,m)^2}{m}a_m \geq 0 \]
hold for all positive divisors $c$ of $N$. It is a cuspidal eta product in and only if all these inequalities hold strict.
\end{proposition}

We also need the following.
\begin{lemma}
\label{lem:higherlevel}
Suppose that an eta product $f(\tau)$ as in \eqref{eq:etaproddef} is holomorphic (resp. cuspidal) for  $\Gamma_0(N)$. Then $f(L\tau)$ is holomorphic (resp. cuspidal) for $\Gamma_0(NL)$ for any integer $L > 0$.
\end{lemma}
\begin{proof}
Let $c|NL$ be a divisor. It can be written as $c=c_1c_2$ where $c_1|N$ and $c_2|L$ (this decomposition is not unique). Then the order of $f(L\tau)$ at $r=\frac{1}{c}$ is
\[ 
\begin{aligned}
\mathrm{ord}(f(L\tau),r)& =\frac{1}{24}\sum_{m} \frac{(c,Lm)^2}{Lm}a_m =\frac{1}{24}\sum_{m} \frac{(c_1c_2,Lm)^2}{Lm}a_m\\
& = \frac{1}{24}\sum_{m} \frac{(c_1,m)^2(c_2,L)^2}{Lm}a_m = \frac{1}{24}\sum_{m} \frac{(c_1,m)^2c_2^2}{Lm}a_m \\
& = \left(\frac{1}{24}\sum_{m} \frac{(c_1,m)^2}{m}a_m\right)\cdot \frac{c_2^2}{L}  =\mathrm{ord}\left(f(\tau),\frac{1}{c_1}\right)\cdot \frac{c_2^2}{L}.
\end{aligned}
\]
By the assumption on $f(\tau)$, this is nonnegative (positive).
\end{proof}

\subsection{Modularity of the eta product expressions for Kleinian orbifolds}

Recall the following result.
\begin{theorem}[{\cite[Theorem 1.2]{bryan2019g}}] 
\label{thm:ZDeltaeta}	
	The $G_{\Delta}$-fixed partition functions can be expressed as eta products as follows.
		\begin{enumerate}
		\item \[Z_{A_n}(\tau)= \frac{1}{\eta(\tau)}\quad\textrm{ for } n\geq 1.\]
		\item
		\[Z_{D_n}(\tau)= \frac{\eta^2(2\tau)\eta((4n-8)\tau)}{\eta(\tau)\eta(4\tau)\eta^2((2n-4)\tau)} \quad\textrm{ for } n \geq 4. \]
		\item \[Z_{E_6}(\tau)= \frac{\eta^2(2\tau)\eta(24\tau)}{\eta(\tau)\eta^2(8\tau)\eta(12\tau)}.\]
		\item \[Z_{E_7}(\tau)= \frac{\eta^2(2\tau)\eta(48\tau)}{\eta(\tau)\eta(12\tau)\eta(16\tau)\eta(24\tau)}.\]
		\item \[Z_{E_8}(\tau)= \frac{\eta^2(2\tau)\eta(120\tau)}{\eta(\tau)\eta(24\tau)\eta(40\tau)\eta(60\tau)}.\]
	\end{enumerate}
\end{theorem}
\begin{remark}
	The coefficients in the eta products of Theorem~\ref{thm:ZDeltaeta} can be interpreted in terms of combinatorial numbers associated with the finite subgroup $G_{\Delta} < SL(2,\SZ)$ corresponding to $\Delta$. For the details see \cite[Section 3.2]{bryan2019g}.
\end{remark}
\begin{corollary}[{\cite[Theorem 1.3]{bryan2019g}}]
\label{cor:zloceta1}
The generating function $Z_{\Delta}(\tau)$ can be written as
\[ Z_{\Delta}(\tau)=\frac{\eta_{\Delta}(\tau)}{\eta^{n+1}(k\tau)} \]
where
		\begin{enumerate}
		\item \[\eta_{A_n}(\tau)= \frac{\eta^{n+1}{((n+1)\tau)}}{\eta(\tau)}\quad\textrm{ for } n\geq 1,\]
		\item
		\[\eta_{D_n}(\tau)= \frac{\eta^2(2\tau)\eta^{n+2}((4n-8)\tau)}{\eta(\tau)\eta(4\tau)\eta^2((2n-4)\tau)} \quad\textrm{ for } n \geq 4, \]
		\item \[\eta_{E_6}(\tau)= \frac{\eta^2(2\tau)\eta^{8}(24\tau)}{\eta(\tau)\eta^2(8\tau)\eta(12\tau)},\]
		\item \[\eta_{E_7}(\tau)= \frac{\eta^2(2\tau)\eta^{9}(48\tau)}{\eta(\tau)\eta(12\tau)\eta(16\tau)\eta(24\tau)},\]
		\item \[\eta_{E_8}(\tau)= \frac{\eta^2(2\tau)\eta^{10}(120\tau)}{\eta(\tau)\eta(24\tau)\eta(40\tau)\eta(60\tau)}.\]
	\end{enumerate}
\end{corollary}
As mentioned in Section~\ref{sec:intro}, this shows also that $R_{\Delta}(\tau)=\eta_{\Delta}(\tau)$.

\begin{example} For $\Delta=A_1$, $\eta_{A_1}(\tau)=\frac{\eta^2(2\tau)}{\eta(\tau)}$ which is a noncuspidal holomorphic modular form of weight $1/2$ and level 2 \cite[Example 3.12 (1)]{kohler2011eta}.\end{example}

As mentioned above, to investigate holomorphicity of $\eta_{\Delta}(\tau)$ for some $\Gamma_0(N)$ it is enough to check the orders at the cusps of the form $r=\frac{1}{c}$ where $c \vert N$. The level $N$ for each $\eta_{\Delta}(\tau)$ equals $k=|G_{\Delta}|$. %As we are interested in holomorphicity for the products of these functions, we need to check the nonnegativitiy of the orders at all cusps of each $\Gamma_0(LN)$ where $L$ is a nonnegative integer.
Although we do a case-by-case analysis in a series of Lemmas, the calculations in each case are very similar. We always apply Lemma~\ref{lem:cusporder}~\eqref{it:cusporderdiv} on the eta quotients in Corollary~\ref{cor:zloceta1}.
\begin{lemma}
\label{lem:an}
The eta product $\eta_{A_n}(\tau)$ is holomorphic for $\Gamma_0(N)=\Gamma_0(n+1)$.
\end{lemma}
\begin{proof}
The order of $\eta_{A_n}(\tau)$ at the cusp $\frac{1}{c}$ for a positive divisor $c \vert (n+1)$  is
\[ \mathrm{ord}\left(\eta_{A_n},\frac{1}{c}\right) =\frac{1}{24}\left( \frac{(c,n+1)^2}{n+1}(n+1) - 1\right)=\frac{1}{24}(c^2-1). \]
This number is nonnegative for all positive $c \vert (n+1)$.
\end{proof}
\begin{lemma}
	The eta product $\eta_{D_n}(\tau)$ is holomorphic for $\Gamma_0(N)=\Gamma_0(4n-8)$.
\end{lemma}
\begin{proof}
The order of $\eta_{D_n}(\tau)$ at the cusp $\frac{1}{c}$ for a positive divisor $c \vert 4n-8$  is
\[ 
\begin{aligned}
\mathrm{ord}\left(\eta_{D_n},\frac{1}{c}\right) & = \frac{1}{24}\left( \frac{(c,4n-8)^2(n+2)}{4n-8}+ (c,2)^2 - \frac{(c,4)^2}{4}  - \frac{(c,2n-4)^2}{n-2} -1 \right) \\
& = \frac{1}{24}\left( \frac{(c,4n-8)^2(n+2)}{4(n-2)}+ (c,2)^2 - \frac{(c,4)^2}{4}  - \frac{(c,2n-4)^2}{n-2} -1 \right). 
\end{aligned}
\]
As $(c,4n-8)^2 \geq (c,2n-4)^2$ and $n \geq 4$, we always have that
\[  \frac{(c,4n-8)^2}{(n-2)}\cdot \frac{n}{4} \geq \frac{(c,2n-4)^2}{(n-2)}. \]
We also have that $(c,4n-8)^2 \geq (c,4)^2$, which implies that
\[ \frac{(c,4n-8)^2}{(n-2)} \cdot \frac{2}{4} \geq \frac{(c,4)^2}{4}. \]
Finally, $(c,2)^2 \geq 1$. We obtain the statement.
\end{proof}
\begin{lemma}
	The eta product $\eta_{E_6}(\tau)$ is holomorphic for $\Gamma_0(N)=\Gamma_0(24)$.
\end{lemma}
\begin{proof}
The order of $\eta_{E_6}(\tau)$ at the cusp $\frac{1}{c}$ for a positive divisor $c \vert 24$  is
	\[ 
		\mathrm{ord}\left(\eta_{E_6},\frac{1}{c}\right) = \frac{1}{24}\left( \frac{8(c,24)^2}{24}+ (c,2)^2 - \frac{2(c,8)^2}{8}  - \frac{(c,12)^2}{12} -1 \right).
	\]
The proof is similar to the type D case.
The facts that $(c,24)^2 \geq (c,12)^2$ and $(c,24)^2 \geq (c,8)^2$  imply that
\[ \frac{2(c,24)^2}{24} \geq \frac{(c,12)^2}{12} \quad\textrm{and}\quad \frac{6(c,24)^2}{24} \geq \frac{(c,8)^2}{4}. \]
Additionally, $(c,2)^2 \geq 1$ as before.
\end{proof}

\begin{lemma}
	The eta product $\eta_{E_7}(\tau)$ is holomorphic for $\Gamma_0(N)=\Gamma_0(48)$.
\end{lemma}
\begin{proof}
	The order of $\eta_{E_7}(\tau)$ at the cusp $\frac{1}{c}$ for a positive divisor $c \vert 48$  is
	\[ 
	\mathrm{ord}\left(\eta_{E_7},\frac{1}{c}\right) = \frac{1}{24}\left( \frac{9(c,48)^2}{48}+ (c,2)^2 - \frac{(c,12)^2}{12} - \frac{(c,16)^2}{16}  - \frac{(c,24)^2}{24} -1 \right).
	\]
	Due to $(c,48)^2 \geq (c,12)^2, (c,16)^2, (c,24)^2$  we have that
	\[ \frac{4(c,48)^2}{48} \geq \frac{(c,12)^2}{12} \quad\textrm{and}\quad \frac{3(c,48)^2}{48} \geq \frac{(c,16)^2}{16}\quad\textrm{and}\quad \frac{2(c,48)^2}{48} \geq \frac{(c,24)^2}{24}. \]
	Additionally, $(c,2)^2 \geq 1$ as before.
\end{proof}

\begin{lemma}
\label{lem:e8}
	The eta product $\eta_{E_8}(\tau)$ is holomorphic for $\Gamma_0(N)=\Gamma_0(120)$.
\end{lemma}
\begin{proof}
	The order of $\eta_{E_8}(\tau)$ at the cusp $\frac{1}{c}$ for a positive divisor $c \vert 120$  is
	\[ 
	\mathrm{ord}\left(\eta_{E_8},\frac{1}{c}\right) = \frac{1}{24}\left( \frac{10(c,120)^2}{120}+ (c,2)^2 - \frac{(c,24)^2}{24} - \frac{(c,40)^2}{40}  - \frac{(c,60)^2}{60} -1 \right).
	\]
	Due to $(c,120)^2 \geq (c,24)^2, (c,40)^2, (c,60)^2$  we have that
	\[ \frac{5(c,120)^2}{120} \geq \frac{(c,24)^2}{24} \quad\textrm{and}\quad \frac{3(c,120)^2}{120} \geq \frac{(c,40)^2}{40}\quad\textrm{and}\quad \frac{2(c,120)^2}{120} \geq \frac{(c,60)^2}{60}. \]
	Additionally, $(c,2)^2 \geq 1$ as before.
\end{proof}

All these prove the following, and give also the main part of Theorem~\ref{thm:maineta}.
\begin{corollary} 
\label{cor:etaorder}	
Let $\Delta$ be a root system of ADE type. Let $n$ be the rank of $\Delta$, and let $k=|G_{\Delta}|$. Then $\eta_{\Delta}$ is a holomorphic modular form of weight $\frac{n}{2}$ for $\Gamma_0(k)$ with order 0 at the cusp 1, with order $\frac{1}{24}((n+1)k-1)$ at $\infty$, and with a positive order at cusps corresponding to every other divisor of $k$.
\end{corollary}
\begin{proof}
The statement on the orders can be checked from the formulas in Lemma~\ref{lem:an}--\ref{lem:e8}.
For example, in Type A, $(c^2-1)$ is zero when $c=1$, and positive otherwise. The order at $\infty$ is $\frac{1}{24}\sum_m ma_m$, which is $\frac{1}{24}((n+1)k-1)$ in each case.
\end{proof}

\begin{remark}
We collected the orders at the various cusps for the type E cases in Appendix~\ref{sec:orderE}. As all entries are non-negative, this gives another, computational proof of the statement in type E.
\end{remark}

%We remark that the order $\frac{1}{24}((n+1)k-1)$ at $\infty$ in each case equals $\frac{(\zeta|\zeta)}{2k}$ due to \eqref{eq:strange}.

%It is also possible to compute the multiplier system of $\eta_{\Delta}$ explicitly. 
The multiplier system of $\eta_{\Delta}$ can also be obtained quickly.
\begin{corollary} 
\label{cor:multetaD}	
The multiplier system $\chi_{\Delta}(A)$ of $\eta_{\Delta}$ is
	
\[ \chi_{\Delta}(A)=\begin{cases}
%	\begin{split}
		(\eta_{\Delta})^\ast(A)e\left(\frac{1}{24}\left(bd((n+1)k-1)\right)\right), & \textrm{if }c\textrm{ is odd,}
%	\end{split}
	\\
%	\begin{split}
		(\eta_{\Delta})_\ast(A)e\left(\frac{1}{24}\left(bd((n+1)k-1)
		+3(d-1)n\right)\right), & \textrm{if }c\textrm{ is even.}
%	\end{split}
\end{cases}
\]
\end{corollary}
\begin{proof}
We observe that for each $\eta_{\Delta}$ in Corollary~\ref{cor:zloceta1}, $\sum a_m=n$ and $\sum_m \frac{a_m}{m}=0$, the order of the cusp 1. Moreover, $\sum_m ma_m=(n+1)k-1$. Combining these facts with Corollary~\ref{cor:fmultexp} we obtain the statement.	
\end{proof}

\clearpage
\appendix

\begin{comment}
\section{Tables}

\subsection{The numbers $m$.}
	\label{subsec:congorder}

\begin{center}
%\begin{table}
	\begin{tabular}{ |c|c|c| }
		\hline
		$\Delta$ & $n$ & $m$ \\
		\hline
		$A_n$ & odd & $2(n+1)$ \\
		& even & $n+1$ \\
		$D_n$ & odd & 8 \\
		& $n \equiv 2\;(\mathrm{mod}\;4)$ & 4 \\
		& $n \equiv 0\;(\mathrm{mod}\;4)$ & 2 \\
		$E_6$ & & 3 \\
		$E_7$ & & 4 \\
		$E_8$ & & 1 \\
		\hline
	\end{tabular}
	\vspace{0.2in}
%	\caption{The numbers $m$.}
%	\label{table:congorder}
%\end{table}
\end{center}
\end{comment}

\section{The orders of $R_{\Delta}(\tau)$ in type E}
\label{sec:orderE}

\begin{center}
	\bgroup
	\def\arraystretch{1.5}
	\begin{tabular}{|c| c|| c| c|| c| c|}
		\hline
		\multicolumn{2}{|c||}{$E_6$} & \multicolumn{2}{c||}{$E_7$}  & \multicolumn{2}{c|}{$E_8$}  \\
		\hline
		$c$ & $\mathrm{ord}$ & 	$c$ & $\mathrm{ord}$ & 	$c$ & $\mathrm{ord}$ \\
		\hline
		$1$ & $0$ & $1$ & $0$& $1$ & $0$ \\
		\hline
		$2$ & $\frac{1}{8}$ & $2$ & $\frac{1}{8}$& $2$& $\frac{1}{8}$\\
		\hline
		$3$ & $\frac{1}{12}$ & $3$&$\frac{1}{48}$ & $3$ & $\frac{1}{120}$ \\
		\hline
		$4$ & $\frac{1}{8}$ & $4$&$\frac{1}{8}$ & $4$ & $\frac{1}{8}$ \\
		\hline
		$6$ & $\frac{11}{24}$ & $6$&$\frac{5}{24}$ & $5$ & $\frac{1}{24}$ \\
		\hline
		$8$ & $\frac{7}{24}$ & $8$&$\frac{7}{24}$ & $6$ & $\frac{19}{120}$ \\
		\hline
		$12$ & $\frac{35}{24}$ & $12$&$\frac{11}{24}$ & $8$ & $\frac{19}{120}$ \\
		\hline
		$24$ & $\frac{167}{24}$ & $16$&$\frac{31}{24}$ & $10$ & $\frac{7}{24}$ \\
		\hline
		& & $24$&$\frac{71}{24}$ & $12$ & $\frac{31}{120}$ \\
		\hline
		& & $48$&$\frac{383}{24}$ & $15$ & $\frac{7}{12}$\\
		\hline
		& & & & $20$ & $\frac{19}{24}$ \\
		\hline
		& & & & $24$ & $\frac{23}{24}$ \\
		\hline
		& & & & $30$ & $\frac{59}{24}$ \\
		\hline
		& & & & $40$ & $\frac{89}{24}$ \\
		\hline
		& & & & $60$ & $\frac{227}{24}$ \\
		\hline
		& & & & $120$ & $\frac{1079}{24}$ \\
		\hline
		
	\end{tabular}
	\egroup
\end{center}

\bibliographystyle{amsplain}
\bibliography{rigid}

\end{document}